\documentclass[12pt]{amsart}

\usepackage{amsmath}
\usepackage{epsfig, graphicx}
\usepackage{pinlabel}

\newtheorem{thm}{Theorem}

\newtheorem{lemma}{Lemma}

\newtheorem*{thma}{Theorem}

\theoremstyle{definition}

\newtheorem{remark}{Remark}

\newtheorem{defn}{Definition}

\newtheorem*{ack}{Acknowledgement}

\def\a{\alpha}
\def\b{\beta}
\def\e{\epsilon}
\def\bR{\mathbb R}

\def\g{\gamma}

\def\rk{\textup{rank}\, }

\input xy
\xyoption{all}

\input epsf
\def\bx{{\bf x}}

\begin{document}

\title{On certain hyperplane arrangements and colored graphs}
\date{\today}
\author{Joungmin Song}
\address{Division of Liberal Arts and Sciences\\GIST\\ Gwangju, 61005, Korea}
\email{songj@gist.ac.kr}
%\begin{abstract}
%\end{abstract}

\begin{abstract}
We exhibit a one-to-one correspondence between $3$-colored graphs and subarrangements of certain hyperplane arrangements denoted $\mathcal J_n$, $n \in \mathbb N$. We define the notion of centrality of $3$-colored graphs, which corresponds to the centrality of hyperplane arrangements. Via the correspondence, the characteristic polynomial $\chi_{\mathcal J_n}$ of $\mathcal J_n$ can be expressed in terms of the number of central $3$-colored graphs, and we compute $\chi_{\mathcal J_n}$ for $n = 2, 3$.
\end{abstract}

\maketitle

\section{Introduction}
A hyperplane arrangement is a finite set of affine hyperplanes in a real affine space. In this article, we shall consider a hyperplane arrangement problem of a specific type: Given a positive integer $n$, let $[n]$ denote $\{1, 2, 3, \dots, n\}$. For each $1 \le \a < \b \le n$, we define
\[
H_{\a\b} := \{ \bx \in \bR^n \, | \, x_\a + x_\b = 1\} = H_{\b\a}
\]
which are said to be walls or hyperplanes of {\bf type I}. For each $i \in [n]$, define
\[
0_i := \{\bx \in \bR^n  \, | \, x_i = 0\}, \mbox{ and } 1_i:=\{\bx \in \bR^n  \, | \, x_i = 1\}
\]
which are said to be walls of {\bf type II}. Let $\mathcal J_n$ denote the hyperplane arrangement consisting of all hyperplanes of type I or type II. We are interested in the number of the {\it regions}, i.e., the connected components of $\mathbb R^n\setminus \bigcup_{H \in \mathcal J_n}H$.

%Our goal is to find the number of (bounded) chambers in the complement of $\cup_{H \in \mathcal J_n} H$.

Our particular hyperplane arrangement has its origin in algebraic geometry, namely the number of certain moduli spaces \cite[Problem~5.2]{Hassett}. But as we were researching for a solution, we discovered a neat relation with graph theory and soon this interplay has taken the center stage. The problem was then modified to better suit the graph theoretic approach. The arrangement $\mathcal J_n$ should remind experts of (deformations of) the braid arrangement, especially  the  well known {\it Shi arrangement} \cite{Shi} which consists of walls of the forms $x_i - x_j = 0$ and $x_i - x_j = 1$. We hope to explore the relation to Shi arrangements in the future.

We  refer to Lecture 1 of Stanley's  chapter \cite{Stanley} on hyperplane arrangements for many fundamental results, but recall a few key notions here. Let $\mathcal B = \{H_i \, | \, i\in I\}$ be a hyperplane arrangement where
\[
{ H_{i} = \left\{ \bx \in \bR^n \, \left| \, \sum_{j=1}^n a_{ij} x_j = b_{i}\right.\right\}}
\]
and $I$ is an index set.
The arrangement $\mathcal{B}$ is said to be {\it central} if the intersection of all hyperplanes in $\mathcal{B}$ is nonempty.
%\[ \bigcap_{i,j\in I(\mathcal{B})} H_{i,j} \bigcap_{i\in I(\mathcal{B})}  0_i  \bigcap_{j\in I(\mathcal{B})} 1_j\ne \emptyset \]
The {\it rank} of a hyperplane arrangement is the dimension of the space spanned by the normal vectors to the hyperplanes in the arrangement.

\begin{defn}\label{D:ass-matrix} Let $A$ and $b$ denote the matrices $(a_{ij})$ and $b = (b_{i}), \ i=1, \dots, |\mathcal B|$, respectively. The augmented matrix $B = [A | b]$ will be called the {\it matrix associated with the hyperplane arrangement $\mathcal B$} or simply the {\it associated matrix of $\mathcal B$}.
\end{defn}
Since the row vectors of $A$ are the normal vectors to the hyperplanes, rank of $\mathcal B$ equals $\rk(A) $. In central hyperplane arrangements, this also equals $\rk(B)$, since the associated matrix is consistent.  Note that $[A | b]$ completely determines $\mathcal B$, and giving a hyperplane arrangement is equivalent to giving its associated matrix.

  Let $\mathcal A$ be a hyperplane arrangement. Then the {\it characteristic polynomial} of $\mathcal A$ is defined
\[
\chi_{\mathcal A}(t) = \sum_{\mathcal B} (-1)^{\left|\mathcal B\right|} t^{n-\rk(\mathcal B)}
\]
where $\mathcal B$ runs through all central sub-arrangements of $\mathcal A$: In fact, the characteristic polynomial is defined using the M\"obius function \cite[Definition~1.3]{Stanley} and the equivalence is a theorem due to H. Whitney \cite[Lemma~2.3]{Orlik-Terao}, but this form suits our purpose just fine.  The following is perhaps the most fundamental theorem when it comes to counting the number of regions.

\begin{thma}\cite{Zaslavsky} Let $\mathcal A$ be a hyperplane arrangement in an $n$-dimensional real vector space. Let $r(\mathcal A)$ be the number of chambers and $b(\mathcal A)$  be the number of relatively bounded chambers. Then we have
\begin{enumerate}
\item $b(\mathcal A) = (-1)^{n}\chi(+1)$.
\item $r(\mathcal A)= (-1)^{n}\chi(-1)$.
\end{enumerate}
\end{thma}

In our case, we have $\binom n2$ walls of type I and $2n$ walls of type II.  All together there are $N = \binom n2 + 2n$ walls and a simple case by case analysis would require centrality examination and rank computation of $2^N$ subarrangements.
The main theme of this paper is that, by a systematic use of symmetry and geometry, we can reduce that number significantly. Enumeration is further enhanced by using the notion of {\it associated graphs} (Definition~\ref{D:ass-graph}) and associated matrices.

The graph theory approach is for the specific hyperplane arrangement problem studied in this paper. We associate a graph to each hyperplane subarrangement of $\mathcal J_n$, and we translate the centrality of  hyperplane arrangements in terms of graph properties (Definition~\ref{D:central} and Theorem~\ref{T:central}). This makes enumeration of central subarrangements much more systematic and efficient: We work out the basic examples of two and three dimensional cases in Section~\ref{S:lowdim} but without borrowing any significant results from graph theory. We hope to employ more substantial  graph theory results to attack the higher dimensional cases in the future. Also, we believe that the method can be generalized to treat other hyperplane arrangements by considering the {\it signed graphs} \cite{Zas-Signed}, and this will be taken up in a forthcoming paper. We thank the anonymous referee for pointing this out.

\begin{ack} The anonymous referee reviewed the article in great detail and made numerous corrections and suggestions. They enormously improved this article and the author would like to thank him/her deeply. The author was supported by GIST Research Fund. The author also would like to thank D. Hyeon for suggesting this problem.
\end{ack}

\section{Associated colored graphs}
Let $(V, E)$ be a graph with vertices $V = [n]$ and the set of edges $E$. Let $v$ and $v'$ be vertices (which may be equal).    A path is said to be {\it even} (resp., odd) if its length is even (resp., odd).

For the purpose of this paper, we shall consider $\{0, 1, *\}$-colored graphs on $V = [n]$ ($*$ indicating no numeric value assigned). We shall let $\gamma : V \to \{0, 1,*\}$ denote the color function. A vertex $v$ with $\gamma(v) = \ast$ will be called {\it not colored}. By a $3$-colored graph, we shall always mean a $\{0, 1, *\}$-colored graph.

\begin{defn} \label{D:central} A $3$-colored graph $([n], E)$ is said to be ${\it central}$ if
\begin{enumerate}
\item if $v$ is colored, then it is not on a  closed walk of odd length, and
\item $\gamma(v) = \gamma(v')$ (resp., $\gamma(v) \ne \gamma(v')$) for any pair of colored vertices $v, v'$ such that there is a $v-v'$ path of even (resp., odd) length.
 \end{enumerate}
\end{defn}

\begin{remark}\label{R:central}
\begin{enumerate}
\item Note that the centrality condition determines the parity of the paths between any two given colored vertices: In a central graph, if a $v-v'$ path is even (resp. odd), so are all other $v-v'$ paths.

\item Condition (1) of Definition~\ref{D:central} is equivalent to that every component $C$ with a colored vertex is bipartite. Suppose $C$ has a closed walk $\gamma$. Since $C$ is connected, there exists a path $\tau$ from a colored vertex $v$ to $\gamma$. Then traversing $\tau$, $\gamma$, and $\tau^{-1}$ back to $v$ is an odd cycle which contains $v$, violating Condition (1). Converse is obvious.
\end{enumerate}
\end{remark}

\

%There is a notion of {\it signed graph} (see \cite{Zas12} for instance) where each edge is assigned a sign of $+$ or $-$.
%Start with a $3$-colored graph $\Gamma$.  Delete all non-colored (i.e. $*$-colored) vertices and all edges adjacent to them, and then take the dual graph of it. Then we obtain a graph with each edge assigned a value of $0$ or $1$. Once we replace $(0,1)$ by $(+,-)
%$, we get a signed graph $\Gamma^\vee$, and the centrality of the $3$-colored graph $\Gamma$ coincides with the {\it balance} (\cite[Section~2.1.1]{Zas12}) of the dual signed graph $\Gamma^\vee$.  It also follows from this correspondence  (Harary's Balance Theorem, \cite[Theorem~2.1]{Zas12}) that

%\begin{proposition} $\Gamma$ is central if and only if $\Gamma^\vee$ is balanced.
%\end{proposition}

%There is a natural one-to-one correspondence between hyperplane arrangements consisting of walls of types I and II and $3$-colored graphs on $[n]$ satisfying a set of conditions.

\begin{defn} For a subarrangement $\mathcal A \subseteq \mathcal J_n$, $I(\mathcal A)$ denotes the set of indices $\tau \in [n]$ that appear in $\mathcal A$. That is,
\[
\displaystyle{
I(\mathcal A) = \left(\bigcup_{H_{\a\b} \in \mathcal A} \{\a, \b\}\right) \cup \left(\bigcup_{0_\a\in \mathcal A} \{\a\}\right) \cup \left( \bigcup_{1_\b \in \mathcal A} \{\b\}\right).}
\]
\end{defn}

\begin{defn}\label{D:ass-graph} Let $\mathcal A$ be a subarrangement of $\mathcal J_n$ which does not contain both $0_i$ and $1_i$, $\forall i$. The {\it associated graph $\Gamma_{\mathcal A}$} of $\mathcal A$  is a $3$-colored graph with the vertex set $V(\Gamma_\mathcal{A})= I(\mathcal{A})$, and edge set $E(\Gamma_\mathcal{A})=\{\{\alpha,\beta\}: H_{\alpha\beta} \in \mathcal{A}\}$
where the vertices are assigned {\bf exactly} one of $\{0, 1, *\}$ in the obvious fashion: given $i\in I(\mathcal A)$, we assign $0$ (resp., $1$) to $i$ if $0_i \in \mathcal A$ (resp., $1_i \in \mathcal A$). If neither $0_i$ nor $1_i$ is in $\mathcal A$, % or they both are in $\mathcal A$,
we assign $\ast$ to $i$, i.e.,  the vertex $i$ is not colored.
\end{defn}

Let $\mathcal S_n$ be the set of all subarrangements of the hyperplane arrangements $\mathcal J_n$ satisfying the conditions of the Definition~\ref{D:ass-graph}, in particular,  those arrangements which do not contain both $0_i$ and $1_i$, $\forall i$.  Then we have the following lemma:
\begin{lemma}
%Let $\mathcal S_n$ be the set of all subarrangements of the hyperplane arrangements $\mathcal J_n$ satisfying the condition of the Definition~\ref{D:ass-graph}.
There is a one-to-one correspondence between $\mathcal S_n$ and the set of all $3$-colored graphs on $[n]$.
\end{lemma}
\begin{proof} Given any $3$-colored graph $\Gamma = ([n], E)$ with the color function $\g : [n] \to \{0,1,*\}$, we associate the hyperplane arrangement
\[
\mathcal A_{\Gamma} :=\left\{ \{x_i = \g(i)\}_{\g(i)\ne *} \right\}\cup \{ H_{ij} \}_{\{i,j\} \in E}.
\]
This clearly is inverse to the association $\mathcal A \mapsto \Gamma_{\mathcal A}$ defined
in Definition~\ref{D:ass-graph}.
\end{proof}

\noindent Note that for the graph $\Gamma$ associated with a central arrangement $\mathcal{A}$, there are no isolated non-colored vertices and two adjacent colored vertices are of different colors.

\begin{thm}\label{T:central} Let $\mathcal A  \in \mathcal S_n.$  Then $\mathcal A$ is central if and only if $\Gamma_{\mathcal A}$ is central.
\end{thm}

\begin{proof}
The corresponding colored graph  $\Gamma_{\mathcal A}$ can be decomposed into  three subgraphs $\Gamma',\ \Gamma'',\ \Gamma'''$
\begin{enumerate}
\item (graph of the first kind)   $\Gamma'$ is the union of colorless connected components; %In particular, no vertex of $\Gamma'$ is   colored;
\item (graph of the second kind) $\Gamma''$ is the  union of isolated colored vertices;% maximal totally disconnected subgraph (i.e. maximal subgraph without edges) all of whose vertices are colored.
\item (graph of the third kind) $\Gamma''' = \Gamma\setminus(\Gamma'\cup \Gamma'')$ is the union of the connected components with at least one colored vertex and at least one edge. %  It is the maximal subgraph such that all of its vertices are incident on an edge, and for each of its non-colored vertices $v$, there are a colored vertex $v'$ of $\Gamma'''$ and a path from $v$ to $v'$.
\end{enumerate}

\def\cA{\mathcal A}

Accordingly, we can decompose $\mathcal A$ into
$ \mathcal A' \sqcup \mathcal A'' \sqcup \mathcal A'''$
such that the subarrangements $\mathcal A', \cA'', \cA'''$ correspond to $\Gamma',\ \Gamma'',\ \Gamma'''$, respectively.
For example, the hyperplane arrangement (Figure 1)
\[
\mathcal A = \{ H_{12}, H_{23}, 1_3, 0_5, H_{46}, 1_7, H_{8,9}\}
\]
decomposes into
\[
 \{  H_{46}, H_{89}\} \sqcup \{  0_5,  1_7\} \sqcup  \{H_{12},  H_{23}, 1_3\}
 \]
whose associated colored graph decomposition is
\begin{enumerate}
\item $\Gamma'= ( \{ 4, 6, 8, 9\}, \{ \{4,6\},\{8,9\}\})$;
\item $\Gamma''= (\{5, 7\}, \emptyset),\ \g(5) = 0,\ \g(7) = +1$;
\item $\Gamma''' = (\{1, 2,3\},  \{ \{1,2\},\{2,3\} \}),\ \g(3) = +1$.

\end{enumerate}

%\begin{figure}%[ht!]
%\labellist \small\hair 2pt
%\pinlabel $1$ at 90 478
%\pinlabel $2$ at 150 478
%\pinlabel $3$ at 212 478
%\pinlabel $8$ at 90 430
%\pinlabel $9$ at 150 430

%\pinlabel $4$ at 212 430
%\pinlabel $7$ at 90 386
%\pinlabel $6$ at 150 386
%\pinlabel $5$ at 212 386
%\endlabellist
%  \begin{center}
%    \includegraphics[height=5cm]{ex2-2}
 % \end{center}
   %\caption{} \label{F:ex2}
%\end{figure}

\begin{figure}[htbp!]
\labellist \small\hair 2pt
\pinlabel $1$ at 90 478
\pinlabel $1$ at 90 478
\pinlabel $2$ at 150 478
\pinlabel $3$ at 212 478
\pinlabel $8$ at 90 432
\pinlabel $9$ at 150 432
\pinlabel $4$ at 212 432
\pinlabel $7$ at 90 388
\pinlabel $6$ at 150 388
\pinlabel $5$ at 212 388

\endlabellist
\begin{center}
\includegraphics[scale=0.8]{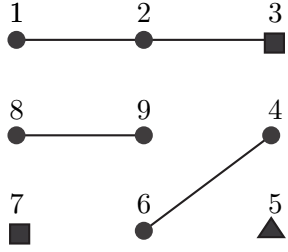}
\caption{Disks denote $*$-colored vertices; squares, 1-colored; and triangles, 0-colored}
\label{default}
\end{center}
\end{figure}

\noindent   Since the index sets $I(\mathcal A')$, $I(\mathcal A'')$ and $I(\mathcal A''')$
 are all disjoint from each other, $\mathcal A$ is central if and only if each of them is separately central.
It is also easy to see that $\mathcal A'$ and its associated graph $\Gamma_{\mathcal A'}$ are central: $\Gamma_{\mathcal A'}$ is trivially central since the centrality condition is vacuous for a non-colored graph. As for the hyperplane arrangement $\mathcal A'$, setting all variables $x_i$, $i \in I(\mathcal A')$, equal to $1/2$ satisfies all hyperplane equations in $\mathcal A'$.  Note that $\mathcal A''$ is also always central except the obvious non-central ones containing both $0_i$ and $1_i$ for some $i$, which we have discarded in the beginning (definition of $\mathcal S_n$).
Hence, without loss of generality, we may assume that $\mathcal A = \mathcal A'''$ so that its associated graph is connected and each vertex is connected to a distinguished colored vertex $i_0$ by a path, and that $\g(i_0) = 1$.

Now, we shall prove the assertion of the theorem for any such hyperplane arrangement. Suppose that the associated graph is central. We define a new coloring $\g^\star$ as follows. If $j$ is colored, set $\g^\star(j) = \g(j)$.
For a non-colored vertex $j$, choose a path $\sigma$ from $j$ to $i_0$ and assign $\g^\star(j) = \g(i_0) = 1$ if the path is even and $\g^\star(j) = 0$ otherwise. Note that all paths from $j$ to $i_0$ have the same parity since two paths of different parity would give a closed walk of odd length from $i_0$ to $i_0$, and this violates the centrality of $\Gamma_{\mathcal A}$. Let $x^\star$ be a point whose coordinates satisfy $x^\star_j = \g^\star(j)$ for any $j \in I(\mathcal A)$. We claim that $x^\star$ satisfies all hyperplane equations. Let $H_{\a\b} \in \mathcal A$. Suppose there is an even path $\sigma$ from $\a$ to $i_0$. By the defining property, $x^\star_\a = 1$. Then joining the edge $\{\b, \a\}$ to $\sigma$ creates an odd path from $\beta$ to $i_0$ which implies that $x^\star_\beta = 0$. Thus $x^\star_\a + x^\star_\b = 1$. The other case where $\g(\a) = 0$ is proved similarly.

Conversely, suppose that $\mathcal A$ is central.
If there is an even path $\{v_0, v_1, \dots, v_{2k}\}$ and $\gamma(v_0) = 0$, then the hyperplanes $x_{v_k} + x_{v_{k+1}} = 1$ corresponding to edges successively determine $x_{v_1} = 1$, $x_{v_2} = 0$, etc., so that $x_{v_{2k}} = 0 = \gamma(v_0)$. Likewise, if there is an odd path  $\{v_0, v_1, \dots, v_{2k+1}\}$, we have $\gamma(v_0) \ne \gamma(v_{2k+1})$. It follows that any two paths between two colored vertices have the same parity.

\end{proof}

\section{Arrangements in the plane and space}\label{S:lowdim}
In this section, we use our main theorem and  analyze the dimension $2$ and $3$ cases. We present them here while promising deeper results in a forthcoming paper where we use graph theoretic approach to give a formula for computing the characteristic polynomial in terms of the number of bipartite graphs of given rank and size.

\begin{defn} \begin{enumerate}
\item We let $r_{\e,\nu}^n$ denote the number of central graphs on $[n]$ with precisely $\e$ edges and $\nu$ colored vertices. When there is no danger of confusion, we drop the superscript $n$.
\item A graph with $\e$ edges and $\nu$ colored vertices will be called an $(\e,\nu)$ graph.
\end{enumerate}
\end{defn}
Note that $\e, \nu$ are the numbers of type I and of type II hyperplanes respectively, hence the order $(\e, \nu)$.

\subsection{ $n=2$}

$\mathcal J_2 = \{H_{12}, 0_1, 1_1, 0_2, 1_2\}$. The central subarrangements $\mathcal B$ are enumerated as follows, according to the cardinality.

\begin{enumerate}
\item $|\mathcal B|=1$:  $\mathcal B$ contains either a single colored vertex, or an edge.  There are five such cases.
\item $|\mathcal B| = 2$: $\{H_{12}, 1_1\}$, $\{H_{12}, 0_1\}$, $\{H_{12}, 1_2\}$, $\{H_{12}, 0_2\}$, $\{0_1, 0_2\}$, $\{1_1, 0_2\}$, $\{0_1,1_2\}$, $\{1_1,1_2\}$. These are of rank $2$.
\item $|\mathcal B|=3$: $\{H_{12}, 1_1, 0_2\}$, $\{H_{12}, 1_2, 0_1\}$. These are of rank $2$.
\item $|\mathcal B| \ge 4$: There are no central subarrangements with $\ge 4$ hyperplanes.
\end{enumerate}

With the trivial central arrangement with  $|\mathcal B| = 0$, the characteristic polynomial is
\[
\chi_{\mathcal J_2}(t) =  ((-1)^2 8 + (-1)^3 2) \cdot t^{2-2} + (-1)\cdot 5 t^{2-1} + t^{2-0} = 6 - 5 t + t^2.
\]
The number of chambers is $\chi(-1) = 12$ and the number of bounded chambers is $\chi(1) = 2$.

\begin{figure}[htbp]
\labellist \small\hair 2pt
\pinlabel $0$ at 155 560
\pinlabel $1$ at 248 560
\pinlabel $1$ at 155 650
\endlabellist
  \begin{center}
    \includegraphics[scale=0.5]{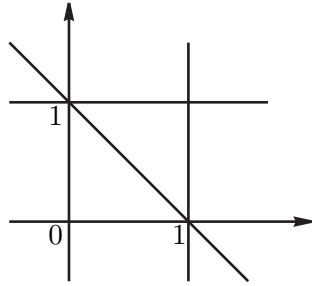}
  \end{center}
   \caption{$\mathcal J_2$ has $12$ regions and $2$ bounded regions.} \label{plane}
\end{figure}

\subsection{$n=3$}

In this case, $\mathcal J_3$ has three hyperplanes of type I and six hyperplanes of type II. Obviously, having both $0_i$ and $1_i$ for a given $i$ makes the arrangement not central. Hence any central subarrangement has no more than three hyperplanes of type II, so a central subarrangement $\mathcal B$ satisfies $|\mathcal B| \le 6$. We enumerate all central graphs, besides the trivial one with $|\mathcal B| = 0$.
\begin{enumerate}
\item $|\mathcal B| = 1 $:  $\mathcal B$ either contains a single colored vertex, or an edge.  There are nine such cases.
\item $|\mathcal B| = 2$: $\mathcal B$ may have { two edges and no colored vertices}, or an edge and a colored vertex, or two colored vertices.
These are all rank 2 arrangements.

%Hence there are $3\times 8 + \binom 32\times 2\times 2 = 36$ central subarrangements with $|\mathcal B| = 2$. These are all of rank $2$.
\begin{enumerate}
\item ($\e = 0$): $\binom 32 \cdot 2 \cdot 2 = 12$ subarrangements;
\item ($\e = 1$): There are $\binom 31 \cdot (2+2) = 12$ central $(1,1)$ graphs with the colored vertex incident on the edge. There are six $(1,1)$ graphs with the colored vertex not incident on the edge.

\item($\e =2$): These are three labeled trees on $[3]$.

\end{enumerate}

\item $|\mathcal B| = 3$: These are all of rank $3$ except the six $(1,2)$ graphs with an edge with both vertices colored (with opposing colors) which are of rank 2.
\begin{enumerate}
\item ($\e = 0$): $2^3$ cases, obviously;
\item ($\e=1$): Choose an edge and either color the two vertices incident upon the edge ($\binom 32 \cdot 2 = 6$ cases, this is essentially the case of two vertices) or color one vertex incident on the edge and the other vertex not on the edge ($\binom 32\cdot 2^3 = 24$ cases);

\item ($\e=2$): Two edges and one colored vertex (of either color). There are $\binom 32 \cdot 3 \cdot 2 = 18$ cases;

\item ($\e=3$): There is only $1$ case.

\end{enumerate}

\item $|\mathcal B| = 4$: $\e \ge 1$ since otherwise a vertex must have two colors. These are all of rank $3,$ and there are 30 central graphs corresponding to $\mathcal B$ with $|\mathcal B|=4$.
\begin{enumerate}
\item $(\e = 1)$: There are two ways to color the two vertices incident on the one existing edge, and two ways to color the vertex not incident on the edge. Hence $\binom 31 \cdot 2 \cdot 2 = 12$ central $(1,3)$ graphs.

\item $(\e = 2)$: There are two edges incident on two colored vertices and a non-colored vertex.  The one non-colored vertex can be at an endpoint of the path or not.  There are $\binom 32  \cdot 2 \cdot 2 = 12$ and $\binom 32\cdot 2 = 6$ central $(2,2)$ graphs of each kind.

\item $(\e = 3)$: There are no central graphs of this type.  A closed path from the colored vertex to itself is a cycle of length 3.

\end{enumerate}

\item $|\mathcal B| = 5$:  There are no (3,2) central graphs.  There are six central $(2,3)$ graphs.  % \sout{ and six $(3,2)$ graphs.}

% \end{minipage}\hspace{1cm}

 %  \end{minipage}

%\vspace{cm}
\item There are no central subarrangements with $\ge 6$ hyperplanes since it would have to have either a cycle of length three with colored vertices or a vertex with two colors.
\end{enumerate}

\vspace{.5cm}

We summarize our findings in the table below.  \vspace{.3cm}

\begin{center}\begin{tabular}{|c||ccccc|}
\hline
rank \textbackslash $|\mathcal B| $ &\  \ $1$ & 2 & 3 & 4 & 5  \\  \hline
1 &\  \ 9 & 0&0 &0 &0   \\
2 & \ \ 0& {33} & 6 &0 &0  \\
3 & \  \ 0&0 & 51 & 30 & {6}  \\
\hline
\end{tabular}
\end{center}
\vspace{.2cm}

The characteristic polynomial is
\[\begin{aligned}
\chi_{\tiny{\mathcal J_3}}(t)  & = t^3 - 9t^{3-1} + \left(33 -6\right) t^{3-2} +(-51+30-6)t^{3-3} \\
& = t^3-9t^2+{27}t-{ 27}. \end{aligned}
\]
Interestingly, the polynomials for $n=2,3$ factor into linear forms.  In general, this turns out not to be the case $n\ge 4$ which will be illustrated in our forthcoming manuscript. It would be interesting to understand when the characteristic polynomial factors.

%\bibliographystyle{alpha}
%\bibliography{generate}

%For the remainder of the paper, we let $G = (V, E)$ denote a connected $3$-colored graph on vertices $V = \{1, 2, \dots, n\}$.

\end{document}